\documentclass[reqno,12pt]{amsart}
\usepackage{amsmath,amssymb,color}
\usepackage{amsthm}
\usepackage{comment}
\usepackage{graphicx}

\newtheorem{theorem}{Theorem}
\newtheorem{corollary}{Corollary}

\newtheorem{remark}{Remark}
\newtheorem{lemma}{Lemma}

\def\re{\mathbb{R}}

\def\eps{\varepsilon}

\def\pd{\partial}
\def\ol{\overline}

\def\la{\lambda}

\def\({\left(}
\def\){\right)}

\def\pd{\partial}

\def\|{\Vert}

\begin{document}
\title[Nonlocal one-dimensional boundary blow up problem]{Bifurcation analysis for nonlocal one-dimensional boundary blow up problems}

\author{Kazuki Sato}

\address{
Department of Mathematics, Osaka Metropolitan University \\
3-3-138, Sumiyoshi-ku, Sugimoto-cho, Osaka, Japan \\
}

\email{sf22817a@st.omu.ac.jp \\}

\author{Futoshi Takahashi}

\address{
Department of Mathematics, Osaka Metropolitan University \\
3-3-138, Sumiyoshi-ku, Sugimoto-cho, Osaka, Japan \\
}

\email{futoshi@omu.ac.jp \\}

\begin{abstract}
In this paper, we study one-dimensional boundary blow up problems with Kirchhoff type nonlocal terms on an interval.
We perform a bifurcation analysis on the problems and obtain the precise number of solutions according to the value of the bifurcation parameter.
We also obtain the precise asymptotic formula for solutions for special cases.
\end{abstract}

%34C23: Bifurcation
%37G99: Local and nonlocal bifurcation theory (none of the above)
% 26D10: Inequalities involving derivatives and differential and integral operators
% 26D15: Inequalities for sums, series and integrals
% 35A23: Inequalities involving derivatives and differential and integral operators, inequalities for integrals
% 46E35: Sobolev spaces and other spaces of ``smooth" functions, embedding theorems, trace theorems

\date{\today}

\dedicatory{}

\maketitle

\section{Introduction}

Let $I = (-1, 1) \subset \re$ be an interval.
In this paper, we consider the following one-dimensional nonlocal elliptic problem with boundary blow up condition: 
\begin{equation}
\label{BBP}
	\begin{cases}
	&A\(\| u \|_{q_1}, \| u' \|_{r_1} \) u''(x) = \la B\( \| u \|_{q_2}, \| u' \|_{r_2}\) u^p(x), \quad x \in I, \\
	&u(x) > 0, \quad x \in I, \\
	&\lim_{x \to \pm 1} u(x) = +\infty,
	\end{cases}
\end{equation}
where 
$p > 1$, $\la > 0$ is a given constant, 
and $A(s, t)$ and $B(s, t)$ are positive continuous functions in $(s, t) \in \re_+ \times \re_+$. 
Here and in what follows, $\| \cdot \|_q$ stands for $\| \cdot \|_{L^q(I)}$ for any $q > 0$. 
We call $u$ a {\it solution} of \eqref{BBP} if $u \in C^2(I)$, $u \in L^{\max \{q_1, q_2 \}}(I)$, $u' \in L^{\max \{ r_1, r_2 \}}(I)$, 
and $u$ solves the equation in the classical sense on $I$.
The constant $\la > 0$ in \eqref{BBP} is considered as a bifurcation parameter, 
i.e., the number of solutions of \eqref{BBP} changes according to the value of $\la$.
In the following, we study the exact number of solutions of \eqref{BBP} when $\la$ varies, 
and also investigate the precise asymptotic behavior of solution curves for specified $A(\cdot, \cdot)$, $B(\cdot, \cdot) \in C(\re_+ \times \re_+ ; \re_+)$.

Throughout of the paper, we assume
\begin{equation}
\label{Assumption}
	0 < q_1, q_2 < \frac{p-1}{2} \quad \text{and} \quad 0 < r_1, r_2 < \frac{p-1}{p+1}.
\end{equation}
For $p > 1$, let $U_p \in C^2(I)$ be a unique even solution to the problem
\begin{equation}
\label{Eq:U_p}
	\begin{cases}
	U_p''(x) = U_p^p(x), \quad x \in I = (-1,1), \\
	U_p(x) > 0, \quad x \in I, \\
	\lim_{x \to \pm 1} U_p(x) = +\infty.
	\end{cases}
\end{equation}
The existence of solutions are standard since the nonlinearity $f(s) = s^p$ $(p > 1)$ satisfies the famous Keller-Osserman condition \cite{Keller}, \cite{Osserman}
\[
	\int_a^{\infty} \frac{1}{\sqrt{F(t)}} dt < +\infty \quad \text{for some} \quad a > 0,
\]
where $F(t) = \int_0^t f(s) ds$. 
For the uniqueness, see Proposition 1.8 and Remark 1.10 in \cite{DDGR}, or Theorem 4 in \cite{Wang-Fan}.
In the former paper \cite{Inaba-TF}, we showed that $U_p \in L^q(I)$ for $0 < q < \frac{p-1}{2}$ by the time map method as in \cite{Shibata(JMAA)}.
The same approach yields that $U' \in L^r(I)$ for $0 < r < \frac{p-1}{p+1}$, see Lemma \ref{Lemma:U_p} below.

In this paper, we prove the following:

\begin{theorem}
\label{Theorem:P}
Let $A, B \in C(\re_+ \times \re_+ ; \re_+)$ and assume \eqref{Assumption}.
Consider the system of equations (w.r.t. $(s_1, s_2, t_1, t_2) \in \re_+ \times \re_+ \times \re_+ \times \re_+$)
\begin{equation}
\label{System}
	\begin{cases}
	&s_1^{1-p} = \la \dfrac{B(s_2, t_2)}{A(s_1, t_1)} \| U_p \|_{q_1}^{1-p} \\
	&s_2^{1-p} = \la \dfrac{B(s_2, t_2)}{A(s_1, t_1)} \| U_p \|_{q_2}^{1-p} \\
	&t_1^{1-p} = \la \dfrac{B(s_2, t_2)}{A(s_1, t_1)} \| U_p' \|_{r_1}^{1-p} \\
	&t_2^{1-p} = \la \dfrac{B(s_2, t_2)}{A(s_1, t_1)} \| U_p' \|_{r_2}^{1-p}
	\end{cases}
\end{equation}
where $U_p$ is the unique solution of \eqref{Eq:U_p}.
Then for $\la > 0$, the problem \eqref{BBP} has the same number of solutions 
$(s_1, s_2, t_1, t_2) \in (\re_+)^4$ of the system of equations \eqref{System}.
Also the number of solutions of \eqref{System} is the same as the number of solutions to the equation
(w.r.t $s \in \re_+$)
\begin{equation}
\label{Single}
	g(s) = \la \| U_p \|_{q_1}^{1-p},
\end{equation}
where
\begin{equation}
\label{g(s)}
	g(s) = s^{1-p} 
\frac{A\(s, \dfrac{\| U_p' \|_{r_1}}{\| U_p \|_{q_1}} s \)}{B\(\dfrac{\| U_p \|_{q_2}}{\| U_p \|_{q_1}}s, \dfrac{\| U_p' \|_{r_2}}{\| U_p \|_{q_1}}s \)}
\end{equation}
 
Moreover, any solution $u_{\la}$ of \eqref{BBP} is of the form 
\begin{equation}
\label{u_form}
	u_{\la}(x) = s_1 \frac{U_p(x)}{\| U_p \|_{q_1}} 
	\( = s_2 \frac{U_p(x)}{\| U_p \|_{q_2}} = t_1 \frac{U_p(x)}{\| U_p' \|_{r_1}} = t_2 \frac{U_p(x)}{\| U_p' \|_{r_2}} \).
\end{equation}
where $(s_1, s_2, t_1, t_2) \in (\re_+)^4$ is a solution of \eqref{System}.
\end{theorem}

As corollaries of Theorem \ref{Theorem:P}, we have many bifurcation diagrams for specific $A$ and $B$ in \eqref{BBP}.
For example, 

\begin{corollary}
\label{Cor1}
Assume \eqref{Assumption} and $A(s,t) = s^{p-1} (1+t)$, $B(s,t) = s + t$ in \eqref{BBP}.
Then 
\begin{enumerate}
\item[(i)] if $0 < \la \le \dfrac{\| U_p \|_{q_1}^{p-1} \| U_p' \|_{r_1}}{\| U_p \|_{q_2} + \| U_p' \|_{r_2}}$, there is no solution to \eqref{BBP}, 
\item[(ii)] if $\la > \dfrac{\| U_p \|_{q_1}^{p-1} \| U_p' \|_{r_1}}{\| U_p \|_{q_2} + \| U_p' \|_{r_2}}$, there is a unique solution to \eqref{BBP}. 
\end{enumerate}
Also the unique solution $u_{\la}$ is of the form
\[
	u_{\la}(x) = s_{\la} \frac{U_p(x)}{\| U_p \|_{q_1}} \quad (x \in I)
\]
where
\[
	s_{\la} = \frac{\| U_p \|^p_{q_1}}{\la  \( \| U_p \|_{q_2} + \| U_p' \|_{r_2} \) - \| U_p \|_{q_1}^{p-1} \| U_p' \|_{r_1}}. 
\]
\end{corollary}

\begin{corollary}
\label{Cor2}
Assume \eqref{Assumption} and $A(s, t) = s^p ((t-a)^2 + b)$, $B(s,t) = s + t$ in \eqref{BBP},
where $a, b > 0$.
Then 
\begin{enumerate}
\item[(i)] if $0 < \la < \dfrac{b \| U_p \|_{q_1}^p}{\| U_p \|_{q_2} + \| U_p' \|_{r_2}}$, there is no solution to \eqref{BBP}, 
\item[(ii)] if $\la = \dfrac{b \| U_p \|_{q_1}^p}{\| U_p \|_{q_2} + \| U_p' \|_{r_2}}$, there is a unique solution to \eqref{BBP}, 
\item[(iii)] if $\dfrac{b \| U_p \|_{q_1}^p}{\| U_p \|_{q_2} + \| U_p' \|_{r_2}} < \la < \dfrac{(a^2 +b)\| U_p \|_{q_1}^p}{\| U_p \|_{q_2} + \| U_p' \|_{r_2}}$, there are just two solutions to \eqref{BBP}, 
\item[(iv)] if $\la \ge \dfrac{(a^2 + b)\| U_p \|_{q_1}^p}{\| U_p \|_{q_2} + \| U_p' \|_{r_2}}$,
there is a unique solution to \eqref{BBP}.
\end{enumerate}
Also any solution $u_{\la}$ is of the form
\[
	u_{\la}(x) = s_{\la} \frac{U_p(x)}{\| U_p \|_{q_1}} \quad (x \in I)
\]
where $s_{\la}$ is a positive solution of the quadratic equation (w.r.t. $s > 0$)
\[
	\(\frac{\| U_p' \|_{r_1}}{\| U_p \|_{q_1}} s-a\)^2 + b = \la \frac{\| U_p \|_{q_2}+\| U_p' \|_{r_2}}{\| U_p \|_{q_1}^p}.
\]
\end{corollary}

\begin{corollary}
\label{Cor3}
Assume \eqref{Assumption} and $A(s, t) = 2 + \sin s$, $B(s, t) = t^{1-p}$ in \eqref{BBP}.
Then 
\begin{enumerate}
\item[(i)] if $0 < \la < \| U_p' \|_{r_2}^{p-1}$, there is no solution to \eqref{BBP}, 
\item[(ii)] if $\| U_p' \|_{r_2}^{p-1} \le \la \le 3 \| U_p' \|_{r_2}^{p-1}$, there are infinitely many solutions to \eqref{BBP}. 
\item[(iii)] if $\la > 3 \| U_p' \|_{r_2}^{p-1}$, there is no solution to \eqref{BBP}. 
\end{enumerate}
Also any solution $u_{\la}$ is of the form
\[
	u_{\la}(x) = s_{\la} \frac{U_p(x)}{\| U_p \|_{q_1}} \quad (x \in I)
\]
where $s_{\la}$ is a positive solution of 
\[
	\| U_p' \|_{r_2}^{p-1} (2 + \sin s) = \la. 
\]
\end{corollary}

\begin{corollary}
\label{Cor4}
Assume \eqref{Assumption} and $A(s, t) = e^s$, $B(s, t) = 1$ in \eqref{BBP}.
Then 
\begin{enumerate}
\item[(i)] if $0 < \la < \(\frac{e}{p-1}\)^{p-1} \| U_p \|_{q_1}^{p-1}$, there is no solution to \eqref{BBP}, 
\item[(ii)] if $\la = \(\frac{e}{p-1}\)^{p-1} \| U_p \|_{q_1}^{p-1}$, there is a unique solution to \eqref{BBP}, 
\item[(iii)] if $\la > \(\frac{e}{p-1}\)^{p-1} \| U_p \|_{q_1}^{p-1}$, there are just two solutions $u_{1,\la}$, $u_{2,\la}$ to \eqref{BBP}.
\end{enumerate}
Also we have
\[
	u_{1,\la}(x) = \la^{-\frac{1}{p-1}} \( 1 + \frac{1}{p-1}(1 + o(1)) \la^{-\frac{1}{p-1}} \| U_p \|_{q_1} \) U_p(x) \quad (x \in I)
\]
as $\la \to \infty$, 
and
\[
	u_{2,\la}(x) = \left\{ \log \la + (p-1) (\log \log \la)(1 + o(1)) \right\} \frac{U_p(x)}{\| U_p \|_{q_1}} \quad (x \in I)
\]
as $\la \to \infty$.
\end{corollary}

\vspace{1em}
In the previous paper \cite{Inaba-TF}, we studied the most primitive problem
\[
	\begin{cases}
	&\(\| u \|_{L^q(I)}^q + b \)^r u''(x) = \la u^p(x), \quad x \in I = (-1,1), \\
	&u(x) > 0, \quad x \in I, \\
	&\lim_{x \to \pm 1} u(x) = +\infty,
	\end{cases}
\]
where $p > 1$, $0 < q < \frac{p-1}{2}$, $r > 0$, $b \ge 0$, and $\la > 0$.
For this problem, the precise number of solutions according to the size of the bifurcation parameter $\la > 0$
and the precise asymptotic behavior of solutions as $\la \to \infty$ are obtained so far.
In this paper, we improve the result drastically, in the sense that we can treat much more general nonlocal terms 
represented by continuous functions $A(\cdot, \cdot)$ and $B(\cdot, \cdot)$ of the norms of unknown functions $u$ and $u'$ in appropriate Lebesgue spaces. 

\vspace{1em}
For bifurcation analysis on the Dirichlet problems with Kirchhoff type nonlocal terms of the form 
\[
	\begin{cases}
	&-A\(\| u \|_{q_1}, \| u' \|_{r_1} \) u''(x) = \la B\( \| u \|_{q_2}, \| u' \|_{r_2}\) u^p(x), \quad x \in I, \\
	&u(x) > 0, \quad x \in I, \\
	&u(\pm 1) = 0,
	\end{cases}
\]
where $p > 1$ and $\la > 0$, 
we refer the readers to a series of works by T. Shibata \cite{Shibata(JMAA)}, \cite{Shibata(BVP)}, \cite{Shibata(ANONA)}, \cite{Shibata(QTDS)}. 

\vspace{1em}
The problem \eqref{BBP} is a one-dimensional version of much more general boundary blow up problem with nonlocal terms
\begin{equation*}
%\label{BBP_f}
	\begin{cases}
	&A\( \| u \|_{L^{q_1}(\Omega)}, \| \nabla u \|_{L^{r_1}(\Omega)} \) \Delta u = \la B\( \| u \|_{L^{q_2}(\Omega)}, \| \nabla u \|_{L^{r_2}(\Omega)} \) f(u) \quad \text{in} \ \Omega, \\
%	&A\( \| u \|_{q_1}, \| \nabla u \|_{r_1} \) \Delta u = \la B\( \| u \|_{q_2}, \| \nabla u \|_{r_2} \) f(u) \quad \text{in} \ \Omega, \\
	&u(x) > 0 \quad x \in \Omega, \\
	&u(x) \to +\infty \quad \text{as} \ {\rm dist }(x, \pd\Omega) \to 0,
	\end{cases}
\end{equation*}
where $\Omega$ is a bounded domain in $\re^N$, $N \ge 1$, 
$f$ is a continuous nonlinearity, $\la > 0$, $q_1, q_2, r_1, r_2 > 0$,
%$\| \cdot \|_{q} = \| \cdot \|_{L^q(\Omega)}$, 
and $A, B \in C(\re_+ \times \re_+ ; \re_+)$ are continuous functions.
As far as we know, this problem has not been considered in the literature.
We believe that the study of \eqref{BBP} will be helpful when we treat much more general problem above in the future.

\section{Some computation of $U_p$}

Recall that $U_p$ is the unique solution of the problem \eqref{Eq:U_p} for $p > 1$.

\begin{lemma}
\label{Lemma:U_p}
$U_p \in L^q(I)$ if and only if $0 < q < \frac{p-1}{2}$, and $U_p' \in L^r(I)$ if and only if $0 < r < \frac{p-1}{p+1}$.
Moreover,
\begin{align}
\label{mu_p}
	&\mu_p := \min_{x \in I} U_p(x) = U_p(0) = \( \sqrt{\frac{p+1}{2}} L_p \)^{\frac{2}{p-1}}, \\ 
\label{U_p-norm}
	&\| U_p \|_{L^q(I)}^q = \sqrt{\frac{2}{p+1}}\mu_p^{\frac{2q-p+1}{2}} \mathrm{B}\(\frac{p-2q-1}{2(p+1)},\frac{1}{2}\),\\
\label{U_p'-norm}
	&\| U_p' \|_{L^r(I)}^r = \(\frac{2}{p+1}\)^{\frac{r+1}{2}} \mu_p^{\frac{(p+1)(r-1)}{2} + 1} \mathrm{B}\(\frac{(1-r)(p+1)-2}{2(p+1)},\frac{r+1}{2}\),
\end{align}
where 
\[
	L_p = \int_{1}^{\infty}\frac{dt}{\sqrt{t^{p+1}-1}} = \(\frac{1}{p+1}\) \mathrm{B}\(\frac{p-1}{2(p+1)}, \frac{1}{2}\)
\]
and ${\rm B}(x,y) = \int_0^1 t^{x-1}(1-t)^{y-1} dt$ denotes the Beta function.
\end{lemma}

\begin{proof}
Formulae \eqref{mu_p} and \eqref{U_p-norm} are proven in \cite{Inaba-TF}.
We can prove \eqref{U_p'-norm} similarly.
By a time map method (see Shibata \cite{Shibata(JMAA)}),
we have
\begin{equation}
\label{Uprime}
	U_p'(x) = \sqrt{\frac{2}{p+1}(U_p^{p+1}(x)-\mu_p^{p+1})}, \quad (0 \le x < 1).
\end{equation}
By \eqref{Uprime} and the change of variables $s = U_p(x)$ and $s= \mu_p t$, we compute
\begin{align*}
	\| U_p' \|_{L_r(I)}^r &= 2\int_0^1 |U_p'(x)|^r dx \\
%	&=2\int_{0}^{1}(U_p'(x))^r \frac{U_p'(x)dx}{\sqrt{\frac{2}{p+1}(U_p^{p+1}(x)-\mu_p^{p+1})}}\\
	&\overset{\eqref{Uprime}}{=} 2\int_0^1 \( \sqrt{\frac{2}{p+1}(U_p^{p+1}(x)-\mu_p^{p+1})} \)^{r-1} U_p'(x) dx \\
	&=2\( \frac{2}{p+1} \)^{\frac{r-1}{2}} \int_{\mu_p}^{\infty} \(s^{p+1}-\mu_p^{p+1}\)^{\frac{r-1}{2}} ds \\
	&=2\( \frac{2}{p+1} \)^{\frac{r-1}{2}} \mu_p^{\frac{(r-1)(p+1)}{2}} \int_{1}^{\infty} \(t^{p+1}- 1 \)^{\frac{r-1}{2}} \mu_p dt \\
	&=\(\frac{2}{p+1} \)^{\frac{r-1}{2}+1} \mu_p^{\frac{(r-1)(p+1)}{2} +1} \int_{1}^{\infty} \(t^{p+1}- 1 \)^{\frac{r-1}{2}} dt \\
	&= \(\frac{2}{p+1}\)^{\frac{r+1}{2}} \mu_p^{\frac{(p+1)(r-1)}{2} + 1} \mathrm{B}\(\frac{(1-r)(p+1)-2}{2(p+1)},\frac{r+1}{2}\).
\end{align*}
\end{proof}

\begin{remark}
By a time map method, we can also obtain a pointwise expression of $U_p$:
\begin{align*}
	&U_p(x) = \mu_p F_p^{-1}(L_p x), \\
	&U_p'(x) = L_p \sqrt{ \(F_p^{-1}(L_p x)\)^{p+1} - 1}
\end{align*}
for $x \in I$, where
\[
	F_p(y) = \int_1^y \frac{ds}{\sqrt{s^{p+1}-1}} \quad (y \ge 1)
\]
and $F_p^{-1}$ is the inverse function of monotone increasing function $F_p$.
Note that $F_p(1) = 0$ implies $F_p^{-1}(0) = 1$ and $F_p(\infty) = L_p$ implies $F_p^{-1}(L_p) = \infty$.
\end{remark}

\section{Proof of Theorem \ref{Theorem:P}.}

In this section, we prove Theorem \ref{Theorem:P}.

\begin{proof}
Let $u \in C^2(I)$ be any solution of \eqref{BBP} and put 
\[
	v = \gamma u
\] 
where $\gamma > 0$ be chosen later.
Then
\begin{align*}
	v''(x) &= \gamma u''(x) \overset{\eqref{BBP}}{=} \gamma \la \dfrac{B(\| u \|_{q_2}, \| u' \|_{r_2})}{A(\| u \|_{q_1}, \| u \|_{r_1})} u^p(x) \\
	&= \gamma \la \frac{B}{A} \(\frac{v(x)}{\gamma}\)^p \\
	&= \underbrace{\( \gamma^{1-p} \la \frac{B}{A} \)}_{=1} v^p(x)
\end{align*}
If we take 
\[
	\gamma = \( \la \frac{B(\| u \|_{q_2}, \| u' \|_{r_2})}{A(\| u \|_{q_1}, \| u' \|_{r_1})} \)^{\frac{1}{p-1}},
\]
then $v$ solves
\[
	\begin{cases}
	v''(x) = v^p(x), \quad x \in I = (-1,1), \\
	v(x) > 0, \quad x \in I, \\
	v(\pm 1) = +\infty,
	\end{cases}
\]
and by the uniqueness, $v \equiv U_p$. 
This implies that 
\begin{equation}
\label{u_form}
	u(x) = \gamma^{-1} U_p(x), \quad u'(x) = \gamma^{-1} U_p'(x).
\end{equation}
Define
\[
	s_1 = \| u \|_{q_1}, \quad s_2 = \| u \|_{q_2}, \quad t_1 = \| u' \|_{r_1}, \quad t_2 = \| u' \|_{r_2}.
\]
Then by \eqref{u_form}, we have
\[
	\begin{cases}
	&s_1 = \gamma^{-1} \| U_p \|_{q_1}, \\
	&s_2 = \gamma^{-1} \| U_p \|_{q_2}, \\
	&t_1 = \gamma^{-1} \| U_p' \|_{r_1}, \\
	&t_2 = \gamma^{-1} \| U_p' \|_{r_2},
	\end{cases}
\]
which is equivalent to \eqref{System}.
This shows that 
\[
	(s_1, s_2, t_1, t_2) = (\| u \|_{q_1}, \| u \|_{q_2}, \| u' \|_{r_1}, \| u' \|_{r_2})
\]
is a solution to \eqref{System} and thus
\[
	\sharp \{ u : \text{solutions of \eqref{BBP}} \} \le \sharp \{ (s_1, s_2, t_1, t_2) \in (\re_+)^4 : \text{solutions of \eqref{System}} \},
\]
where $\sharp A$ denotes the cardinality of the set $A$.

On the other hand, let $(s_1, s_2, t_1, t_2) \in (\re_+)^4$ be any solution to \eqref{System}.
Note that by \eqref{System}, we see
\begin{align}
\label{System2}
	&\frac{s_1}{\| U_p \|_{q_1}} = \frac{s_2}{\| U_p \|_{q_2}} = \frac{t_1}{\| U_p' \|_{r_1}} = \frac{t_2}{\| U_p' \|_{r_2}} \\
	&= \left\{ \la \frac{B(s_2, t_2)}{A(s_1, t_1)} \right\}^{\frac{1}{1-p}}. \notag
\end{align}
Thus if we define 
\[
	u(x) = s_1 \frac{U_p(x)}{\| U_p \|_{q_1}} \(= s_2 \frac{U_p(x)}{\| U_p \|_{q_2}} = t_1 \frac{U_p(x)}{\| U_p' \|_{r_1}} = t_2 \frac{U_p(x)}{\| U_p' \|_{r_2}} \),
\]
then we have $u(x) > 0$, $u(\pm 1) = +\infty$, and 
\[
	s_1 = \| u \|_{q_1}, \quad s_2 = \| u \|_{q_2}, \quad t_1 = \| u' \|_{r_1}, \quad t_2 = \| u' \|_{r_2}.
\]
Moreover, by the definition of $u$, we have
\begin{align*}
	A\(\| u \|_{q_1}, \| u' \|_{r_1}\) u''(x) &= A(s_1, t_1) u''(x) \\
	&= A(s_1, t_1) \frac{s_1}{\| U_p \|_{q_1}} U_p''(x) \\
	&\overset{\eqref{Eq:U_p}}{=} A(s_1, t_1) \frac{s_1}{\| U_p \|_{q_1}} U_p^p(x) \\
	&= A(s_1, t_1) \frac{s_1}{\| U_p \|_{q_1}} \(\frac{\| U_p \|_{q_1}}{s_1} u(x) \)^p \\
	&= A(s_1, t_1) s_1^{1-p} \| U_p \|_{q_1}^{p-1} u^p(x) \\
	&\overset{\eqref{System}}= \la B(s_2, t_2) u^p(x).
\end{align*}
This shows that
\[
	\sharp \{ u : \text{solutions of \eqref{BBP}} \} \ge \sharp \{ (s_1, s_2, t_1, t_2) \in (\re_+)^4 : \text{solutions of \eqref{System}} \}.
\]
Thus the number of solutions of \eqref{BBP} and that of \eqref{System} are the same.

Also by \eqref{System2}, we can rewrite the system of equations \eqref{System} into a single equation for $s = s_1$ 
\[
	\frac{s_1}{\| U_p \|_{q_1}}
	= \left\{ \la \dfrac{B\(\dfrac{\| U_p \|_{q_2}}{\| U_p \|_{q_1}} s_1, \dfrac{\| U_p' \|_{r_2}}{\| U_p \|_{q_1}} s_1 \)}
{A\(s_1, \dfrac{\| U_p' \|_{r_1}}{\| U_p \|_{q_1}} s_1 \)} \right\}^{\frac{1}{1-p}} 
\]
which is equivalent to \eqref{Single} with $g(s)$ in \eqref{g(s)}.
Thus the number of solutions of \eqref{System} and that of \eqref{Single} are also the same. 
\end{proof}

\section{Proof of corollaries.}

In this section, we prove corollaries in \S 1.
By Theorem \ref{Theorem:P}, the number of solutions of \eqref{BBP} is the same as the number of solutions of \eqref{Single}.
Also we have the expression of solution $u(x) = s \frac{U_p(x)}{\| U_p \|_{q_1}}$ to \eqref{BBP} where $s > 0$ is any solution of \eqref{Single}. 
Thus we just need to solve the equation \eqref{Single} with $\la > 0$ for specific $A$ and $B$.

\vspace{1em}
{\it Proof of Corollary \ref{Cor1}}:\hspace{1em}
Since $A(s, t) = s^{p-1} (1 + t)$ and $B(s, t) = s + t$, a simple computation shows that the function $g(s)$ defined in \eqref{g(s)} becomes
\[
	g(s) = s^{1-p} 
\dfrac{A\(s, \dfrac{\| U_p' \|_{r_1}}{\| U_p \|_{q_1}} s \)}{B\(\dfrac{\| U_p \|_{q_2}}{\| U_p \|_{q_1}}s, \dfrac{\| U_p' \|_{r_2}}{\| U_p \|_{q_1}}s \)}
= \frac{1}{\| U_p \|_{q_2} + \| U_p' \|_{r_2}} \( \frac{ \| U_p \|_{q_1} + \| U_p' \|_{r_1} s}{s} \).
\]
Note that $g(s)$ is monotonically decreasing for $s > 0$, 
$\lim_{s \to +0} g(s) = +\infty$ and $\lim_{s \to +\infty} g(s) = \frac{\| U_p' \|_{r_1}}{\| U_p \|_{q_2} + \| U_p' \|_{r_2}}$.
Thus the equation \eqref{Single} has a solution if and only if 
\[
	\la \| U_p \|_{q_1}^{1-p} > \frac{\| U_p' \|_{r_1}}{\| U_p \|_{q_2} + \| U_p' \|_{r_2}}
\]
and the unique solution $s_{\la} > 0$ of \eqref{Single} in this case is
\[
	s_{\la} = \frac{\| U_p \|^p_{q_1}}{\la  \( \| U_p \|_{q_2} + \| U_p' \|_{r_2} \) - \| U_p \|_{q_1}^{p-1} \| U_p \|_{r_1}}. 
\]
Note that the solution of \eqref{BBP} (with $A, B$ as above) has the form $u(x) = \la \frac{U_p(x)}{\| U_p \|_{q_1}}$.
This proves Corollary \ref{Cor1}.
\qed

\vspace{1em}
{\it Proof of Corollary \ref{Cor2}}:\hspace{1em}
Since $A(s, t) = s^p \{ (t-a)^2 + b \}$ and $B(s, t) = s + t$, the function $g(s)$ defined in \eqref{g(s)} becomes
\begin{align*}
	g(s) &= s^{1-p} 
\dfrac{A\(s, \dfrac{\| U_p' \|_{r_1}}{\| U_p \|_{q_1}} s \)}{B\(\dfrac{\| U_p \|_{q_2}}{\| U_p \|_{q_1}}s, \dfrac{\| U_p' \|_{r_2}}{\| U_p \|_{q_1}}s \)} \\
&= \frac{\| U_p \|_{q_1}}{\| U_p \|_{q_2} + \| U_p' \|_{r_2}} \left\{ \(\frac{\| U_p' \|_{r_1}}{\| U_p \|_{q_1}} s - a \)^2 + b \right\}.
\end{align*}
Thus $g(s)$ is a quadratic function of $s > 0$ with 
\[
	\min_{s > 0} g(s) = \frac{\| U_p \|_{q_1}}{\| U_p \|_{q_2} + \| U_p' \|_{r_2}} b, \quad g(0) = \frac{\| U_p \|_{q_1}}{\| U_p \|_{q_2} + \| U_p' \|_{r_2}} (a^2 + b).
\]
Now, the equation \eqref{Single} 
\[
	\frac{\| U_p \|_{q_1}}{\| U_p \|_{q_2} + \| U_p' \|_{r_2}} \left\{ \(\frac{\| U_p' \|_{r_1}}{\| U_p \|_{q_1}} s - a \)^2 + b \right\}
	= \la \| U_p \|_{q_1}^{1-p}
\]
is a quadratic equation w.r.t. $s > 0$.
Thus the number of positive solutions of \eqref{Single} according to the value of $\la > 0$ is:
\begin{enumerate}
\item[(i)] $0$, if $0 < \la \| U_p \|_{q_1}^{1-p} < \frac{\| U_p \|_{q_1}}{\| U_p \|_{q_2} + \| U_p' \|_{r_2}} b$,
\item[(ii)] $1$, if $\la \| U_p \|_{q_1}^{1-p} = \frac{\| U_p \|_{q_1}}{\| U_p \|_{q_2} + \| U_p' \|_{r_2}} b$,
\item[(iii)] $2$, if $\frac{\| U_p \|_{q_1}}{\| U_p \|_{q_2} + \| U_p' \|_{r_2}} b < \la \| U_p \|_{q_1}^{1-p} < \frac{\| U_p \|_{q_1}}{\| U_p \|_{q_2} + \| U_p' \|_{r_2}} (a^2 + b)$,
\item[(iv)] $1$, if $\la \| U_p \|_{q_1}^{1-p} \ge \frac{\| U_p \|_{q_1}}{\| U_p \|_{q_2} + \| U_p' \|_{r_2}} (a^2 + b)$.
\end{enumerate}
This completes the proof of Corollary \ref{Cor2}.
\qed

\vspace{1em}
{\it Proof of Corollary \ref{Cor3}}:\hspace{1em}
$A(s, t) = 2 + \sin s$ and $B(s, t) = t^{1-p}$ implies that the function $g(s)$ is
\begin{align*}
	g(s) &= s^{1-p} 
\dfrac{A\(s, \dfrac{\| U_p' \|_{r_1}}{\| U_p \|_{q_1}} s \)}{B\(\dfrac{\| U_p \|_{q_2}}{\| U_p \|_{q_1}}s, \dfrac{\| U_p' \|_{r_2}}{\| U_p \|_{q_1}}s \)} \\
	&= \( \frac{\| U_p' \|_{r_2}}{\| U_p \|_{q_1}} \)^{p-1} (2 + \sin s).
\end{align*}
Thus the equation \eqref{Single} is reduced to
\[
	\| U_p' \|_{r_2}^{p-1} (2 + \sin s) = \la.
\]
By this observation, the result of Corollary \ref{Cor3} is easily obtained.
\qed

\vspace{1em}
{\it Proof of Corollary \ref{Cor4}}:\hspace{1em}
Since $A(s, t) = e^s$ and $B(s, t) = 1$, the function $g(s)$ defined in \eqref{g(s)} becomes
\[
	g(s) = e^s s^{1-p}.
\]
A computation yields that $\min_{s > 0} g(s) = g(p-1) = \(\frac{e}{p-1}\)^{p-1}$ and $\lim_{s \to +0} g(s) = \lim_{s \to \infty} g(s) = +\infty$.
Thus the number of solutions to \eqref{Single} in this case is
\begin{enumerate}
\item[(i)] $0$, if $0 < \la \| U_p \|_{q_1}^{1-p} < \(\frac{e}{p-1}\)^{p-1}$,
\item[(ii)] $1$, if $\la \| U_p \|_{q_1}^{1-p} = \(\frac{e}{p-1}\)^{p-1}$,
\item[(iii)] $2$, if $\la \| U_p \|_{q_1}^{1-p} > \(\frac{e}{p-1}\)^{p-1}$.
\end{enumerate}
Let $0 < s_{1, \la} < s_{2, \la}$ be two solutions in (iii). Then we see $s_{1,\la} \to 0$ and $s_{2,\la} \to \infty$ as $\la \to \infty$.
Since
\[
	e^{s_{i,\la}} = s_{i,\la}^{p-1} \la \| U_p \|_{q_1}^{1-p}
\]
for $i=1,2$, by taking $\log$ of the above equation, we obtain
\begin{equation}
\label{sila}
	s_{i, \la} = (p-1) \log s_{i,\la} + \log \la + (1-p) \log \| U_p \|_{q_1}.
\end{equation}
Since $s_{1,\la} = o(1)$ as $\la \to \infty$, we have
\[
	o(1) = (p-1) \log s_{1,\la} + \log \la + (1-p) \log \| U_p \|_{q_1}
\]
as $\la \to \infty$.
Thus we see
\begin{equation}
\label{s1la}
	s_{1,\la} = \la^{-\frac{1}{p-1}} \| U_p \|_{q_1} e^{o(1)} = \la^{-\frac{1}{p-1}} \| U_p \|_{q_1} (1 + \delta)
\end{equation}
where $\delta \to 0$ as $\la \to \infty$.
Inserting \eqref{s1la} into the both sides of \eqref{sila} and computing, we have
\[
	\la^{-\frac{1}{p-1}} \| U_p \|_{q_1} (1 + \delta) = (p-1) \log ( 1 + \delta) = (p-1) \delta(1 + o(1)).
\]
From this, we have
\begin{align*}
	\la^{-\frac{1}{p-1}} \| U_p \|_{q_1} = (p-1) \frac{\delta}{1 + \delta} (1 + o(1)) = (p-1) \delta (1 + o(1)), \\
\end{align*}
Thus
\begin{align*}
	\delta = \frac{1}{p-1} (1 + o(1)) \la^{-\frac{1}{p-1}} \| U_p \|_{q_1}.
\end{align*}
Inserting this into \eqref{s1la} and recall $u_{1,\la} = s_{1,\la} \frac{U_p(x)}{\| U_p \|_{q_1}}$,
we obtain the formula for $u_{1,\la}$:
\[
	u_{1,\la}(x) = \la^{-\frac{1}{p-1}} \( 1 + \frac{1}{p-1}(1 + o(1)) \la^{-\frac{1}{p-1}} \| U_p \|_{q_1} \) U_p(x) \quad (x \in I).
\]

Similarly, from \eqref{sila}, we have
\[
	s_{2, \la} = (p-1) \log s_{2,\la} + \log \la + (1-p) \log \| U_p \|_{q_1}.
\]
Now, since $s_{2, \la} \to \infty$ as $\la \to \infty$, we have
\[
	s_{2, \la} \(1 - (p-1) \underbrace{\frac{\log s_{2,\la}}{s_{2,\la}}}_{=o(1)} \) = \log \la + (1-p) \log \| U_p \|_{q_1} = \log \la + o(\log \la)
\]
as $\la \to \infty$.
From this, we have
\begin{equation}
\label{s2la}
	s_{2,\la} = \frac{1}{1-o(1)} (\log \la)(1 + o(1)) = (1+ \eps) \log \la
\end{equation}
as $\la \to \infty$, where $\eps \to 0$ as $\la \to \infty$. 
Again, inserting this into \eqref{sila}, we have
\[
	\eps \log \la = (p-1) \log (1 + \eps) + (p-1) \log \log \la + (1-p) \log \| U_p \|_{q_1},
\]
which implies
\[
	\eps = (p-1)\frac{\log \log \la}{\log \la} \left\{ 1 - \frac{\| U_p \|_{q_1}}{\log \log \la} \right\} = (p-1)\frac{\log \log \la}{\log \la} (1 + o(1))
\]
as $\la \to \infty$. Then coming back to \eqref{s2la} and recalling that $u_{2,\la} = s_{2,\la} \frac{U_p(x)}{\| U_p \|_{q_1}}$,
we obtain the asymptotic formula for $u_{2,\la}$:
\[
	u_{2,\la}(x) = \left\{ \log \la + (p-1) (\log \log \la)(1 + o(1)) \right\} \frac{U_p(x)}{\| U_p \|_{q_1}} \quad (x \in I)
\]
as $\la \to \infty$.
\qed

\section{Exponential nonlinearity case.}

In this section, we study the following boundary blow up problem with exponential nonlinearity:
\begin{equation}
\label{BBP_E}
	\begin{cases}
	&A\(\| u' \|_{r_1} \) u''(x) = \la B\(\| u' \|_{r_2}\) e^{u(x)}, \quad x \in I = (-1, 1), \\
	&\lim_{x \to \pm 1} u(x) = +\infty,
	\end{cases}
\end{equation}
where $A, B$ are positive continuous functions, $\la > 0$, and $r_1, r_2 > 0$. 
$u \in C^2(I)$ is a solution if $u' \in L^{\max\{r_1, r_2}(I)\}$ and $u$ satisfies \eqref{BBP_E} pointwisely.
Note that a solution $u$ may change its sign on $I$.
For \eqref{BBP_E}, we assume
\begin{equation}
\label{Assumption_E}
	0< r_1, r_2 < 1.
\end{equation}

To study \eqref{BBP_E}, the unique solution $U_{\la}$ to the problem
\begin{equation}
\label{Eq:U_la}
	\begin{cases}
	U_{\la}''(x) = \la e^{U_{\la}(x)}, \quad x \in I = (-1,1), \\
	\lim_{x \to \pm 1} U_{\la}(x) = +\infty
	\end{cases}
\end{equation}
is very important.
For the existence and the uniqueness of $U_{\la}$ for any $\la > 0$, see \cite{CDG} Theorem 1.2., and the references there in. 

First, we prove the following lemma.

\begin{lemma}
\label{Lemma:U_la}
For any $\la > 0$, let $U_{\la}$ be the unique solution to \eqref{Eq:U_la}.
Then $U_{\la}' \in L^r(I)$ if and only if $0 < r < 1$, and
\begin{align}
\label{mu_la}
	&\mu_{\la} := \min_{x \in I} U_{\la}(x) = U_{\la}(0) = \log \(\frac{\pi^2}{2 \la}\), \\ 
\label{U_la-form}
	&U_{\la}(x) = \log \(\frac{\pi^2}{2 \la}\) - 2 \log \( \cos \( \frac{\pi x}{2} \) \), \quad x \in I, \\ 
\label{U_la-prime-norm}
	&\| U_{\la}' \|_{L_r(I)}^r = 2 \pi^{r-1} \mathrm{B}\(\frac{1-r}{2},\frac{r+1}{2}\).
\end{align}
\end{lemma}

\begin{proof}
By the uniqueness, $U_{\la}$ must be an even function on $I$ and $\mu_{\la} = U_{\la}(0)$.
By a time map method, we have
\begin{equation}
\label{U_la_prime}
	U_{\la}'(x) = \sqrt{2\la \(e^{U_{\la}(x)} - e^{\mu_{\la}}\)}, \quad (0 \le x < 1).
\end{equation}
By \eqref{U_la_prime} and the change of variables $t = U_{\la}(x)$ and $t= \mu_{\la} s$, we compute
\begin{align*}
	1 &= \int_0^1 1 dx =\int_0^1 \frac{U_{\la}'(x)}{\sqrt{2\la \(e^{U_{\la}(x)} - e^{\mu_{\la}}\)}} dx \\
	&=\frac{1}{\sqrt{2\la e^{\mu_{\la}}}} \int_{1}^{\infty} \frac{ds}{s\sqrt{s-1}} = \frac{1}{\sqrt{2\la e^{\mu_{\la}}}} \int_{0}^{\infty} \frac{2}{t^2 + 1} dt,
\end{align*}
which implies
\[
	\sqrt{2\la e^{\mu_{\la}}} = 2 \arctan (t) \Big |_{t=0}^{t=\infty} = \pi.
\]
Thus we obtain \eqref{mu_la}.
Also, for $y \in [0, 1)$, we have
\begin{align*}
	y &= \int_0^y 1 dx =\int_{0}^{y} \frac{U_{\la}'(x)}{\sqrt{2\la \(e^{U_{\la}(x)} - e^{\mu_{\la}}\)}} dx \\
	&=\frac{1}{\sqrt{2\la e^{\mu_{\la}}}} \int_{1}^{e^{U_{\la}(y) - \mu_{\la}}} \frac{ds}{s\sqrt{s-1}} \\
	&=\frac{1}{\sqrt{2\la e^{\mu_{\la}}}} \int_{0}^{\sqrt{\exp(U_{\la}(y) - \mu_{\la}) -1}} \frac{2}{t^2 + 1} dt \\
	&=\frac{2}{\sqrt{2\la e^{\mu_{\la}}}} \arctan \(\sqrt{\exp(U_{\la}(y) - \mu_{\la}) -1}\).
\end{align*}
By inserting $\mu_{\la} = \log \(\frac{\pi^2}{2 \la}\)$ to this expression, we have
\begin{align*}
	y = \frac{2}{\pi} \arctan \( \sqrt{\frac{2\la}{\pi^2} e^{U_{\la}(y)} - 1} \).
%	&y = \frac{2}{\pi} \arctan \( \sqrt{\frac{2\la}{\pi^2} e^{U_{\la}(y)} - 1} \), \\
%	&\tan \(\frac{\pi y}{2} \) = \sqrt{\frac{2\la}{\pi^2} e^{U_{\la}(y)} - 1}, \\
%	&1 + \tan^2 \(\frac{\pi y}{2} \) = \frac{2\la}{\pi^2} e^{U_{\la}(y)}.
\end{align*}
Form this, we obtain \eqref{U_la-form}.
Finally, since
\[
	U_{\la}'(x) = - 2 \( \log \( \cos \( \frac{\pi x}{2} \) \) \)' = \pi \tan \( \frac{\pi}{2} x \),
\]
we have
\begin{align*}
	\| U_{\la}' \|_r^r &= 2 \int_0^1 |U_{\la}'(x)|^r dx = 2 \pi^r \int_0^1 \tan^r \( \frac{\pi}{2} x \) dx \\
	&= 2 \pi^r \int_0^{\infty} s^r \cdot \(\frac{2}{\pi}\) \(\frac{ds}{1+s^2}\) \\
	&= 2 \pi^{r-1} \mathrm{B}\(\frac{1-r}{2},\frac{r+1}{2}\) < \infty
\end{align*}
if and only if $0 < r < 1$. This proves \eqref{U_la-prime-norm}.
\end{proof}

Using Lemma \ref{Lemma:U_la}, we obtain the following.

\begin{theorem}
\label{Thm:E}
Assume \eqref{Assumption_E}.
Then for any $\la > 0$ there is a unique solution $u_{\la}$ to \eqref{BBP_E}.
Moreover, $u_{\la}$ must be of the form
\begin{equation}
\label{u_la}
	u_{\la}(x) = U_{\la}(x) - \log \( \la \frac{B(\| U_{\la}' \|_{r_2}}{A(\| U_{\la}' \|_{r_1}} \) \quad (x \in I),
\end{equation}
where $U_{\la}$ is the unique solution to \eqref{Eq:U_la} given by \eqref{U_la-form} with $\| U_{\la}' \|_{r}$ in \eqref{U_la-prime-norm}. 
\end{theorem}
	
\begin{proof}
For any $\la > 0$, define $u_{\la}$ as in \eqref{u_la}.
Then
\[
	u_{\la}'(x) = U_{\la}'(x), \quad \|u_{\la}'\|_{r_i} = \|U_{\la}'\|_{r_i} \quad (i=1,2).
\]
Thus
\begin{align*}
	A\(\| u' \|_{r_1} \) u_{\la}''(x) &= A\(\| U_{\la}' \|_{r_1} \) U_{\la}''(x) \\
	&\overset{\eqref{Eq:U_la}}{=} A\(\| U_{\la}' \|_{r_1} \) e^{U_{\la}(x)} \\
	&\overset{\eqref{u_la}}{=} A\(\| U_{\la}' \|_{r_1} \) e^{u_{\la}(x) + \log \( \la \frac{B(\| U_{\la}' \|_{r_2}}{A(\| U_{\la}' \|_{r_1}} \)} \\
	&= A\(\| U_{\la}' \|_{r_1} \) \la \frac{B(\| U_{\la}'\|_{r_2})}{A(\| U_{\la}' \|_{r_1})} e^{u_{\la}(x)} \\
	&= \la B\(\| u_{\la}' \|_{r_2} \) e^{u_{\la}(x)}
\end{align*}
and $u_{\la}(\pm 1) = +\infty$. Thus $u_{\la}$ solves \eqref{BBP_E}.

On the other hand, for any solution $u_{\la}$ to \eqref{BBP_E}, define
\begin{equation}
\label{v_la}
	v_{\la}(x) = u_{\la}(x) + \log \( \la \frac{B(\| u_{\la}' \|_{r_2}}{A(\| u_{\la}' \|_{r_1}} \) \quad (x \in I).
\end{equation}
Then 
\begin{align*}
	v_{\la}''(x) &\overset{\eqref{v_la}}{=} u_{\la}''(x) \overset{\eqref{BBP_E}}{=} \la \frac{B\(\| u_{\la}' \|_{r_2}\)}{A\(\| u_{\la}' \|_{r_1}\)} e^{u_{\la}(x)} \\
	&= \la e^{u_{\la}(x) + \log \( \frac{B\(\| u_{\la}' \|_{r_2}\)}{A\(\| u_{\la}' \|_{r_1}\)} \)} \\
	&\overset{\eqref{v_la}}{=} \la e^{v_{\la}(x)}
\end{align*}
and $v(\pm 1) = +\infty$.
Thus $v_{\la}$ solves \eqref{Eq:U_la} and  $v_{\la} \equiv U_{\la}$ by the uniqueness of $U_{\la}$. 
Also $U_{\la}'(x) = v_{\la}'(x) = u_{\la}'(x)$ for any $x \in I$ by \eqref{v_la}. 
Returning to \eqref{v_la}, we see that $u_{\la}$ has the form \eqref{u_la}.
\end{proof}

\vskip 0.5cm

\noindent\textbf{Acknowledgement.} 
The second author (F.T.) was supported by JSPS Grant-in-Aid for Scientific Research (B), No. 23H01084, 
and was partly supported by Osaka Central University Advanced Mathematical Institute (MEXT Joint Usage/Research Center on Mathematics and Theoretical Physics).

\end{document}